\documentclass{amsart}
\usepackage[utf8]{inputenc}
\usepackage{multirow}
\usepackage{amssymb,bm,amsfonts, stmaryrd, amsmath, amsthm, enumerate, mathtools,comment}
\usepackage[all]{xy}
\usepackage{csquotes}
\makeatletter
\@namedef{subjclassname@2020}{\textup{2020} Mathematics Subject Classification}
\makeatother
\usepackage[colorlinks]{hyperref}
\hypersetup{
 colorlinks=true,
 linkcolor=blue,
 filecolor=magenta, 
 urlcolor=cyan,
}
\usepackage{wrapfig,tikz, tikz-cd}
\usetikzlibrary{arrows}
\usepackage[capitalise]{cleveref}
\crefformat{equation}{(#2#1#3)}
\crefrangeformat{equation}{(#3#1#4--#5#2#6)}
\crefformat{enumi}{(#2#1#3)}
\crefrangeformat{enumi}{(#3#1#4--#5#2#6)}
\newtheorem{theorem}{Theorem}[section]
\newtheorem{lemma}[theorem]{Lemma}
\newtheorem{proposition}[theorem]{Proposition}
\newtheorem{corollary}[theorem]{Corollary}
\newcounter{intro}

\newtheorem{introthm}[intro]{Theorem}

\newtheorem{introdef}[intro]{Definition}

\theoremstyle{definition}

\newtheorem{example}[theorem]{Example}
\newtheorem{remark}[theorem]{Remark}

\newtheorem{chunk}[theorem]{}

\newtheorem*{theorem*}{Theorem}
\newtheorem*{definition*}{Definition}
\newtheorem*{lemma*}{Lemma}
\newtheorem*{proposition*}{Proposition}
\newtheorem*{corollary*}{Corollary}

\theoremstyle{definition}

\newcommand{\pdim}{{\operatorname{pdim}}}

\newcommand{\Ext}{{\operatorname{Ext}}}
\newcommand{\Tor}{{\operatorname{Tor}}}

\newcommand{\del}{\partial}
\renewcommand{\H}{\operatorname{H}}

\DeclareMathOperator{\coker}{coker}
\DeclareMathOperator{\depth}{depth}

\newcommand{\V}{{\rm{V}}}

\newcommand{\ann}{{\operatorname{ann}}}

\DeclareMathOperator{\height}{height}

\newcommand{\m}{\mathfrak{m}}

\newcommand{\h}{\mathsf{h}}

\DeclareMathOperator{\Supp}{Supp}

\DeclareMathOperator{\Kos}{Kos}
\DeclareMathOperator{\Fitt}{Fitt}

\newcommand{\ma}{\mathfrak a}

\newcommand{\ldepth}[2]{\operatorname{lift.depth}_{#1}(#2)}
\newcommand{\ldim}[2]{\operatorname{lift.dim}_{#1}(#2)}

\title{Lifting systems for finite length modules}

\author[B. Katz]{Benjamin Katz}
\address{University of Nebraska Lincon, NE 68588. U.S.A.}
\email{bkatz2@huskers.unl.edu}
\urladdr{https://math.unl.edu/person/ben-katz/}

\author[N.~KC] {Nawaj KC}
\address{University of Utah, Salt Lake City, UT 84112. U.S.A.}
\email{nawaj.kc@utah.edu}
\urladdr{https://sites.google.com/view/kcnawaj/home}

\author[K. Mohana Sundaram]{Kesavan Mohana Sundaram}
\address{University of Nebraska Lincon, NE 68588. U.S.A.}
\email{km2@huskers.unl.edu}
\urladdr{https://kesavan-ms.github.io/}

\author[A. Soto Levins]{Andrew J. Soto Levins}
\address{Texas Tech University, TX 79409. U.S.A.}
\email{ansotole@ttu.edu}
\urladdr{https://sites.google.com/view/andrewjsotolevins}

\author[R. Watson]{Ryan Watson}
\address{University of Nebraska Lincon, NE 68588. U.S.A.}
\email{rwatson9@huskers.unl.edu}
\urladdr{https://rawatson1997.github.io}

\date{\today}
\keywords{lifting systems, liftable depth, liftable dimension, finite length modules, liftability, Serre liftability}
\subjclass[2020]{Primary: 13C14, 13C15, 13D22  Secondary: 13D40, 13E15}

\begin{document}

\begin{abstract}
    This paper is concerned with lifting modules along a surjective map of noetherian local rings, say $\varphi \colon R \twoheadrightarrow S$. A finitely generated $R$-module $L$ is a \textit{naive lift} of an $S$-module $M$ if $L \otimes_R S \cong M$. We are concerned with the maximum depth and dimension among all naive lifts of $M$, which we call the \textit{liftable depth} and \textit{liftable dimension}, respectively, of $M$ along $\varphi$. We approach this via a notion of \textit{lifting systems} that we introduce in this paper. We then provide a necessary and sufficient condition for a module of finite length to lift and Serre lift to a regular local ring in terms of lifting systems. 
\end{abstract}
\maketitle
\section{Introduction}

Given a surjection of noetherian local rings $\varphi\colon R \twoheadrightarrow S$ and a finitely generated $S$-module $M$, there exists a finitely generated $R$-module $L$ such that $L \otimes_RS\cong M$; for  example,  we can lift a presentation matrix of $M$ to $R$ and take $L$ to be its cokernel. We call such an $L$ a \textit{naive lift} of $M$. Since there are choices on how we can lift a presentation, an $S$-module may admit many non-isomorphic naive lifts. The problem is to determine how large the depth and dimension of a naive lift of a given $S$-module can be.

That is, we consider two invariants of an $S$-module $M$, with respect to $\varphi$, called \textit{liftable depth} and \textit{liftable dimension}, denoted as $\ldepth{\varphi}{M}$ and $\ldim{\varphi}{M}$. These are defined to be the maximum depth and dimension among all naive lifts of $M$ respectively. For the remainder of this section, we assume that $S$ has finite projective dimension over $R$ and that $M$ is a finite length $S$-module.

 By a formula of Auslander \cite[Theorem 1.2]{Auslander:1961}, it follows that \begin{equation} \label{ldepthbound}
     \ldepth{\varphi}{M} \leq \depth R - \depth S, 
 \end{equation}and a module that achieves this upper bound admits a \textit{lift} to $R$; see \cref{Auslander}. That is, there exists a finitely generated $R$-module $L$ such that $L \otimes_R S \cong M$ and $\Tor^R_i(L, S) = 0$ for $i >0$. In this case, we say $M$ is a \textit{liftable} $S$-module to $R$. Note that when $M$ lifts to some $R$-module $L$, the minimal free resolution of $M$ lifts to the minimal free resolution of $L$. This notion of lifting has been studied in a number of works \cite{Auslander/Ding/Solberg:1993, Buchsbaum/Eisenbud:1972, Dao:2007, Hochster:1975, Jorgensen:1999, Jorgensen:2003, Peskine/Szpiro:1973}.

On the other hand, it seems to be a more difficult problem to determine a uniform upper bound for the liftable dimension, at least at this level of generality. Conjecturally, the answer is 
\begin{equation} \label{ldimbound}
    \ldim{\varphi}{M} \leq \dim R - \dim S,
\end{equation}
which is a consequence of a long standing conjecture of Peskine--Szpiro \cite[Ch.II, 0.2]{Peskine/Szpiro:1973}. The dimension inequality of Serre \cite[Ch.V, Thm.3]{Serre} establishes \cref{ldimbound} when $R$ is regular and the New Intersection theorem \cite{RobertsNIT:1987} establishes it when we assume that $S$ is also Cohen--Macaulay. It is also known in the case when $R$ is a hypersurface due to Hochster \cite{Hochster:1981}. However, beyond these cases, we do not know if \cref{ldimbound} is true. 

By definition, an $S$-module $M$ for which $\ldim{\varphi}{M}$ achieves this conjectural upper bound is said to be \textit{Serre liftable} to $R$; this is a new notion of lifting introduced and studied in \cite{KC;2025, KC/SotoLevins:2024}. This family of modules enjoys some remarkable properties: for example, if a finite length module $M$ over $S$ admits a Serre lift to some unramified regular local ring $R$, then $\ell_S(M) \geq e(S)$, where $e(S)$ denotes the Hilbert--Samuel multiplicity of $S$ \cite[Theorem 2.1]{KC/SotoLevins:2024}.

The connection between classical liftability and Serre liftability lies in the following curious fact: when $R$ is regular or when we are in the standard graded setting with $M$ of finite projective dimension, if $\ldepth{\varphi}{M}$ achieves the maximum possible value in \cref{ldepthbound}, then so does $\ldim{\varphi}{M}$ in \cref{ldimbound}; that is, liftable modules to a regular ring are Serre liftable \cite[Proposition 1.6]{KC/SotoLevins:2024} and likewise in the graded setting \cite[Thm 2.2.8]{KC;2025}. We do not know to what extent the regularity or graded hypothesis is needed; see \cite[Question 2.2.12]{KC;2025}. But there exist unliftable modules, of finite and infinite projective dimension, that admit Serre lifts \cite[Example 1.7 \& 3.2]{KC/SotoLevins:2024} and \cref{unliftablemodules}.

\subsection*{New results}
In this paper we establish lower bounds on liftable depth and liftable dimension using a notion of \textit{lifting systems} for finite length modules. We draw our inspiration from a deformation-theoretic idea of lifting modules used by Auslander--Ding--Solberg in characterizing \textit{liftable modules} along a hypersurface surjection \cite[Theorem 1.2]{Auslander/Ding/Solberg:1993}. The key new idea is the following definition.

\begin{introdef} \label{theDefinition}
    Given a surjection of noetherian local rings $\varphi\colon R \twoheadrightarrow S=R/\ma$ and a finite length $S$-module $M$, a collection of finite length $R$-modules $\{M_n\}_{n\geq 1}$ is a \textit{lifting system} for $M$ along $\varphi$ if  $M_1 = M$, $\ma^n \subseteq \ann_RM_n$, and $M_{n+1}/\ma^nM_{n+1} \cong M_n$ for all $n \geq 1$. The $R$-module $L = \varprojlim M_n$ is called the \textit{lift} associated to the system.
\end{introdef}

This terminology comes from the fact that if $R$ is $\ma$-adically complete and $\{M_n\}_{n \geq 1}$ is a lifting system for $M$ along $\varphi$ with the associated lift $L$, we have that $L \otimes_R S \cong M$. That is, $L$ is a naive lift of the $S$-module $M$; see \cref{liftingprop}. Our main results give lower bounds for depth and dimension of $L$ in terms of invariants depending on the finite length modules in the system $\{M_n\}_{n \geq 1}$.

In the following result, when we write $\ell_R(R/I_n) \approx n^r$, we mean that there are polynomials $f$ and $g$ of degree $r$ such that
\[f(n) \leq \ell_R(R/I_n) \leq g(n)\]
for $n\gg 0$.

\begin{introthm}\label{theTheorem}

    Consider a surjection $\varphi\colon R \twoheadrightarrow S = R/\ma$ of noetherian local rings where $R$ is $\ma$-adically complete and let $r$ and $d$ be nonnegative integers. Then the following hold.
    \begin{enumerate}
    \item If $S$ has finite projective dimension over $R$, then a finite length $S$-module $M$ satisfies $\ldepth{\varphi}{M} \geq r$ if and only if $M$ admits a lifting system $\{ M_n \}_{n \geq 1}$ along $\varphi$ such that $\varprojlim _{n}\Tor^R_i(M_n, S) = 0$ for all $i > \pdim_RS - r$. In this case, $L = \varprojlim M_n$ is a naive lift such that $\depth L \geq r$. 
    \item A finite length $S$-module $M$ satisfies $\ldim{\varphi}{M} \geq r$ if and only if it admits a lifting system $\{ M_n\}_{n \geq 1}$ along $\varphi$ such that $\ell_R(R/I_n) \approx n^d$, where $I_n = \Fitt^R_0(M_n)$ is the zeroth Fitting ideal of $M_n$ over $R$ and $d\geq r$. In this case, $L = \varprojlim M_n$ is a naive lift such that $\dim L \geq r$. Moreover, if $\dim{L}=r$, then $d=r$.
\end{enumerate}
\end{introthm}
 
As corollaries to this theorem, we characterize when a finite length module over a local ring is liftable and when it is Serre liftable to a regular local ring; see \cref{cor1} and \cref{cor2}. Utilizing the methods introduced in this paper, in \Cref{S_Examples} we produce unliftable modules of finite length and finite projective dimension that are Serre liftable and modules that intersect properly; see \Cref{E_UnliftableSerreLiftbleExample} and \Cref{T_IntersectProperlyEquimultiple} respectively. 

\section{Lifting systems}
In this section we establish properties of lifting systems, our main result in this direction is the following. We write $\mu_{R}(-)$ for the minimal number of generators of a finitely generated $R$-module.
\begin{proposition} \label{liftingprop}
    Let $\varphi\colon R \twoheadrightarrow S = R/\ma$ be a surjection of noetherian local rings where $R$ is $\ma$-adically complete and let $M$ be a finite length $S$-module. Assume that $\{M_n\}_{n\geq 1}$ is a lifting system for $M = M_1$ along $\varphi$, and let $L = \varprojlim M_n$ denote the associated \textit{lift} of this system; see \Cref{theDefinition}. Then $L$ satisfies the following: \begin{enumerate}
        \item $L \otimes_R R/\ma^n \cong M_n$ for each $n \geq 1$, and 
        \item $L$ is a finitely generated $R$-module with $\mu_R(L) = \mu_R(M_n)$ for each $n \geq 1$  and 
        \[ \ann_R(L) = \bigcap_{n \geq 1}\ann_R(M_n).\]
    \end{enumerate}
\end{proposition}

We prove this proposition in \cref{A_ProofLiftingSystemProp}. We now describe more explictly the presentation matrices of the modules in a lifting system; see \cref{thisiswhatitlookslike}. 
\begin{chunk} \label{R_LiftingSystemRemark} Suppose that $\varphi\colon R \twoheadrightarrow S = R/\ma$ is a surjection of noetherian local rings. For each $n \geq 1$, let $A_n$ denote the row matrix with the minimal set of generators of $\ma^n$ as its entries. And for $\mu \geq 1$, we write $A_n^{\oplus \mu}$ for the diagonal block matrix of size $\mu$ with $A_n$ on its diagonal. 

Let $M_1 = M$ be an $S$-module with a presentation matrix $\phi$ and suppose that $\mu_S(M) = \mu$. Equivalently, $\phi$ is a $\mu \times m$ matrix over $S$ for some $m \geq 0$. By abusing notation, we also view $\phi$ as a matrix over $R$ by naively lifting its entries. The following is a presentation of $M_1 = M$ as an $R$-module. 
\[
\phi_1 =
\left(
\begin{array}{c|c}
\phi & (A_1)^{\oplus \mu}
\end{array}
\right)
\]
If $M_2$ is an $R/\ma^2$-module such that $M_2/\ma M_{2} \cong M_1$, then we can write a presentation matrix of $M_2$ as an $R$-module as follows.
\[ \phi_2 = \left(
\begin{array}{c|c}
\phi + \sigma_1& (A_2)^{\oplus \mu}
\end{array}
\right),\]
where $\sigma_1$ is a matrix the same size as $\phi$ and its entries are either zero or in $\ma \setminus \ma^2$. Inductively, at the $n$'th stage, we can write a presentation matrix of $M_n$ as an $R$-module as follows 
\[ \phi_n = \left(\begin{array}{c|c}
\phi + \sum_{j=1}^{n-1}\sigma_i& (A_n)^{\oplus \mu}
\end{array}
\right),\] where for each $j$, $\sigma_j$ is a matrix the same size as $\phi$ and its entries are either zero or in $\ma^j \setminus \ma^{j+1}$. We record the following as a lemma for future use. 

\begin{lemma} \label{thisiswhatitlookslike}
    Suppose that $\varphi\colon R \twoheadrightarrow S = R/\ma$ is a surjection of noetherian local rings. If $M$ is an $S$-module with a presentation matrix $\phi$, then a lifting system $\{M_n\}_{n\geq 1}$ for $M$ along $\varphi$ satisfies: \begin{enumerate}
        \item For each $n \geq 1$, the following is a presentation matrix for $M_n$ as an $R$-module 
        \[ \phi_n = \left(
\begin{array}{c|c}
\phi + \sum_{j=1}^{n-1}\sigma_j& (A_n)^{\oplus \mu}
\end{array}
\right), \]
where $\mu = \mu_S(M)$, and for each $j \geq 1$, $\sigma_j$ is a matrix the same size as $\phi$ and its entries are either zero or in $\ma^j \setminus \ma^{j+1}$.
    \item If $\ma = (g_1, \ldots, g_d)$, then for each $j \geq 1$, we can write 
    \[
    \sigma_j = \sum_{k=1}^d \sigma_{j,k} g_k\] 
    where $\sigma_{j, k}$ is a matrix the same size as $\sigma_j$ with each entry in $\ma^{j-1}$ for all $j \geq 1$ and $1 \leq k \leq d$.
    \item The following matrix \[ \Phi \coloneqq \phi + \sum_{j=1}^{\infty
    } \sigma_j \] is well-defined over $\widehat{R}^{\mathfrak a}$ and $L= \coker \Phi \cong \varprojlim M_n$.
    \end{enumerate}
\end{lemma}
    \begin{proof}
        The discussion preceding this lemma shows (1). For (2), we note that each entry of $\sigma_j$ is either zero or in $\ma^j$, so we may write each of its entries as a sum $\sum_{k = 1}^d r_i g_i$ where each $r_i \in \ma^{j-1}.$ This suffices to show (2).

        For (3), clearly $L$ is a finitely generated $\widehat{R}^{\ma}$-module and thus $L$ is $\ma$-adically complete. Therefore, \[ L \cong \varprojlim L/\ma^n L \cong \varprojlim M_n,\] where the first isomorphism is because $L$ is $\ma$-adically complete and second isomorphism is because $L/\ma^n L \cong \coker(\phi_n) \cong M_n$. 
    \end{proof}
\end{chunk}

The following is an explicit example demonstrating the conclusion of this lemma. 
\begin{example} \label{workingexample}
      Consider $ \varphi\colon R \to S = R/\ma$ where $R = k[[x, y, z]]$ and $\ma = (x^2,y^2)$. Let $M_1 = M$ be the $S$-module that has the following as a presentation matrix  
    \[ \phi = \left(\begin{array}{cc}
         z & y  \\
         y & z
    \end{array} \right).\] 
    For each $j \geq 1$, set 
    \[\sigma_j = \left(\begin{array}{cc}
         0 & 0  \\
         x^{2j} & 0
    \end{array} \right) \text{ if }j\text{ is even, and }
    \sigma_j = \left(\begin{array}{cc}
         0 & x^{2j}  \\
         0 & 0
    \end{array} \right) \text{ if }j\text{ is odd.}
    \] 
    Clearly, for any $j \geq 1$, $\sigma_j$ has entries that either zero or in $\ma^j \setminus \ma^{j+1}$. Now as an $R$-module, $\phi_n = \left(\begin{array}{c|c}
\phi + \sum_{j=1}^{n-1}\sigma_j& (A_n)^{\oplus 2}
\end{array} \right)$ is a presentation matrix for $M_n$. Writing this explicitly, for all $n\geq 3$, when $n$ is even, $M_n$ is the cokernel of the map  \[ R^{2n + 4} \xrightarrow{\left[\begin{array}{cc|cccccccc}
        z & s& x^{2n} & x^{2n-2}y^2 & \cdots 
        & y^{2n} & 0 & 0 & \cdots 
        & 0 \\
       t & z & 0 & 0 & \cdots 
        & 0 & x^{2n} & x^{2n-2}y^2 & \cdots 
        & y^{2n} 
    \end{array}\right]} R^{2},\] where $s = y + x^2 + x^6 + \cdots + x^{2n-2}$ and $t = y + x^4 + x^8 + \cdots + x^{2n-4}$. Likewise, when $n$ is odd, $M_n$ is the cokernel of the map \[ R^{2n + 4} \xrightarrow{\left[\begin{array}{cc|cccccccc}
      z & s& x^{2n} & x^{2n-2}y^2 & \cdots 
        & y^{2n} & 0 & 0 & \cdots 
        & 0 \\
       t & z & 0 & 0 & \cdots 
        & 0 & x^{2n} & x^{2n-2}y^2 & \cdots 
        & y^{2n} 
    \end{array}\right]} R^{2},\] where $s = y + x^2 + x^6 + \cdots + x^{2n-4}$ and $t = y + x^4 + x^8 + \cdots + x^{2n-2}.$ Finally, the associated lift of the lifting system considered is the $R$-module $L$ which is isomorphic to the cokernel of \[ R^{2} \xrightarrow{\left[\begin{array}{cc}
        z & y + \sum_{i=1}^{\infty} x^{4i-2} \\
        y + \sum_{i=1}^{\infty} x^{4i} & z
    \end{array} \right]} R^2.\]
\end{example}

\section{Liftable depth}
In this section we prove \Cref{theTheorem}(1) and \Cref{cor1}. First we recall a formula of Auslander and a case for when Tor commutes with inverse limits. For the next two remarks, $\varphi\colon R\rightarrow S=R/\ma$ is a surjection of noetherian local rings.

\begin{chunk} \label{Auslander}
    Let $M$ and $N$ be finitely generated $R$-modules. If $M$ has finite projective dimension over $R$ and $M \otimes_RN$ has finite length, 
    Auslander's depth formula \cite[Theorem 1.2]{Auslander:1961} establishes 
    \[ \depth N = \pdim_R M - q_R(M, N),\]
    where $q_R(M, N) = \sup\{i : \Tor^R_{i}(M, N) \not= 0 \}.$ Therefore, if $S$ has finite projective dimension over $R$, then for any finite length $S$-module $M$ one has 
    \[ \ldepth{\varphi}{M}\leq\pdim_RS = \depth R - \depth S,\]
    and $\ldepth{\varphi}{M} = \pdim_RS$ if and only if $M$ lifts to $R$. 
\end{chunk}

\begin{chunk}
    Suppose that $\{M_n\}_{n\geq1}$ is a lifting system along $\varphi$ for some finite length $S$-module $M$. We can then consider the system of $R$-complexes $\{M_n \otimes_R F\}_{n\geq 1}$ where $F$ is a minimal free resolution of the $R$-module $S$. Since ${M_n}$ is a system of surjective maps, the induced map on complexes is degree wise surjective. Now applying \cite[Theorem 3.5.8]{Weibel:1994}, we see that \begin{equation} \label{invlimitiso}
        \varprojlim \Tor^R_i(M_n, S) \cong \Tor^R_i(\varprojlim M_n, S).
    \end{equation} 
    Here we have also used the fact that for all $i \geq 0$ one has $\varprojlim^{(1)} \Tor^R_i(M_n, S) = 0$ as the Tor groups are finite length modules; see \cite[Proposition 3.5.7]{Weibel:1994}. 
\end{chunk}

\begin{proof}[Proof of \cref{theTheorem}(1).]
    For $L$ a finitely generated naive lift of $M$ with $\depth L \geq r$, we apply Auslander's formula to $L \otimes_R S \cong M$ and observe that
    \[\pdim_R S - r\geq q_R(L, S).\]
    Clearly $\{M_n:=L/\ma^nL\}_{n \geq 1}$ is a lifting system of $M$ along $\varphi$, with associated lift $\widehat{L}^{\ma}$, the $\ma$-adic completion of $L$. Since $R$ is $\ma$-adically complete and since $L$ is a finitely generated $R$-module by \Cref{liftingprop}(2), we have $\widehat{L}^{\ma} \cong L$, and so
    \[\varprojlim \Tor^R_i(M_n, S) \cong \Tor^R_i(\varprojlim M_n, S) \cong \Tor^R_i(L, S) = 0\]
    for $i>q_{R}(L,S)$, where the first isomorphism is by \cref{invlimitiso}.

    Conversely, given a lifting system $\{M_n\}_{n \geq 1}$ such that $\varprojlim \Tor^R_i(M_n, S)=0$ for $i>\pdim_{R}S-r$, then $L = \varprojlim M_n$ satisfies $L \otimes_R S \cong M$ by \cref{liftingprop}(1) and it satisfies $\depth L \geq r$ since
    \[\pdim_{R}S - r \geq q_{R}(L,S) = \pdim_{R}S - \depth{L},\]
    where the inequality is by assumption and the equality is by Auslander's formula \Cref{Auslander}. We can apply Auslander's formula here because $L$ is a finitely generated $R$-module by \cref{liftingprop}(2). This completes the proof.
\end{proof}

\begin{corollary} \label{cor1}
    Let $\varphi\colon R \twoheadrightarrow S = R/\ma$ be a surjection of noetherian local rings where $R$ is regular and $\ma$-adically complete. A finite length $S$-module $M$ lifts to $R$ if and only if it admits a lifting system $\{M_n\}_{n\geq 1}$ along $\varphi$ such that $\varprojlim \Tor^R_1(M_n, S) = 0$. 
\end{corollary}
\begin{proof}
The $S$-module $M$ lifting to a finitely generated $R$-module $L$ gives the second equality below
\[\pdim_{R}S -\depth{L} = q_{R}(L,S) = 0,\]
and the first equality is by Auslander's formula. The liftable depth of $M$ is at least $\depth{L}$, hence by \Cref{theTheorem}(1) there is a lifting system $\{ M_n \}_{n\geq 1}$ such that 
\[ \varprojlim\Tor^R_i(M_n, S) = 0\] 
for all $i > \pdim_R S - \depth{L}=0$.

For the converse, let $L = \varprojlim M_n$ be the lift associated to the lifting system $\{M_n\}_{n\geq 1}$. Then the assumption gives the second equality below
\[\Tor_{1}^{R}(L,S) = \Tor_{1}^{R}(\varprojlim M_n,S) \cong \varprojlim \Tor^R_1(M_n, S) = 0\]
and \Cref{invlimitiso} gives the isomorphism. Modules over regular local rings are Tor rigid by \cite{Lichtenbaum:1966}, and so $q_{R}(L,S)=0$. Since we also have that $L$ is a naive lift of $M$ by \cref{liftingprop}(1), $L$ is a lift of $M$
\end{proof}

In the following example we compute the lifting depth of $M$ from \cref{workingexample}.
\begin{example} \label{workingexamplefordepth}
Retaining the setup of \cref{workingexample}, the complex 
\[0\rightarrow R^{2} \xrightarrow{\left[\begin{array}{cc}
        z & y + \sum_{i=1}^{\infty} x^{4i-2} \\
        y + \sum_{i=1}^{\infty} x^{4i} & z
    \end{array} \right]} R^2 \rightarrow 0\]
is a minimal free resolution of $L$ over $R$. This is because the determinant of the matrix above is a nonzerodivisor in $R$, and so the Buchsbaum-Eisenbud acyclicity criterion implies that the complex is exact, see \cite[Theorem 1.4.13]{Bruns/Herzog:1998}. This implies $\depth{L}=2$ by the Auslander-Buchsbaum formula. Now \cref{ldepthbound} gives the second inequality below
\[2 = \depth{L} \leq \ldepth{\varphi}{M} \leq \depth{R} - \depth{S} = 2.\]
\end{example}

\section{Liftable dimension}
In this section, we provide a proof for \cref{theTheorem}(2) and \cref{cor2}. Our key insight is looking at the zeroth Fitting ideals of the modules in the lifting system. 
\begin{chunk}\label{Fittingideals}
    For basic facts about Fitting ideals, we refer to \cite[20.2]{Eisenbud:1995}. Recall that if $R$ is a noetherian local ring and $M$ is a finitely presented $R$-module with a presentation 
    \[ R^{\beta_1} \xrightarrow{\phi} R^{\beta_0} \to M \to 0,\]
    one defines $\Fitt^R_i(M) = I_{\beta_0-i}(\phi)$, the ideal generated by the $(\beta_0-i)$-minors of the matrix $\phi$. Note that for an ideal $I$ in $R$, we have an equality: $\Fitt_{0}(R/I)=I$. The ideals $\Fitt^R_i(M)$ do not depend on the choice of presentation of $M$, see \cite[20.4]{Eisenbud:1995}. Furthermore, for a finitely generated $R$-module $M$, we have 
    \[ \ann_R(M)^{\mu_{R}(M)} \subseteq \Fitt^R_0(M) \subseteq \ann_R(M),\]
    whence $\Supp_R(M) = \V(\ann_R(M)) = \V(\Fitt^R_0(M))$, see \cite[Prop 20.7]{Eisenbud:1995}. 
\end{chunk}

\begin{lemma} \label{fittingLemma}
    Suppose that $R \twoheadrightarrow S = R/\ma$ is a surjection of noetherian local rings. If $L$ is a finitely generated $R$-module generated by $\mu$ elements, then \[\Fitt_0^R(L) + \ma^{\mu} \subseteq \Fitt^R_0(L \otimes_R S) \subseteq \Fitt^R_0(L) + \ma. \]
\end{lemma}
\begin{proof}
    Suppose that $R^{\beta_1} \xrightarrow{\phi} R^{\beta_0} \to L \to 0$ is a presentation matrix for $L$. If $\ma = (a_1, \ldots, a_n)$, define a matrix 
    \[ A = [a_1I_{\beta_0}  \ldots  a_nI_{\beta_0}],\]
    where $I_{\beta_0}$ is the identity matrix of size $\beta_0$. Then note that $\phi' = (\phi | A)$ is a presentation matrix for $L \otimes_R S$ over $R$. It is now easy to see that \[ \Fitt^R_0(L \otimes_R S) = \sum_{j = 0}^{\mu} \ma^j\Fitt^R_j(L).\] Noting that $\Fitt^R_{\mu}(L) = R$, the assertion now follows readily. 
\end{proof}

We record the following fundamental theorem from Hilbert--Samuel theory as a lemma for future use, see \cite[Section 13]{Matsumura:1989}.  
\begin{lemma} \label{lemma}
    Suppose that $(R, \m)$ is a noetherian local ring. If $\mathfrak b$ and $\mathfrak c$ are ideals such that $\sqrt{\mathfrak b + \mathfrak c} = \m$, then \[ \ell_R(R/(\mathfrak b + \mathfrak c^n)) \approx n^{\dim R/\mathfrak b}.\]
\end{lemma}

\begin{proof}[Proof of \cref{theTheorem}(2)]
First suppose that there exists a finitely generated $R$-module $L$ such that $L \otimes_R S \cong M$ and $\dim L \geq r$. Then $\{ M_n \}_{n \geq 1}=\{ L/\ma^n L\}_{n\geq 1}$ is a lifting system for $M$. Let $I = \Fitt^R_0(L)$ and $I_n = \Fitt^R_0(M_n)$ for all $n \geq 1$. By \cref{fittingLemma}, we have the following containments for all $n \geq 1$ 
\begin{equation} \label{containment}
    I + \ma^{n\mu} \subseteq I_n \subseteq I + \ma^n,
\end{equation}
where $\mu=\mu_{R}(L)$. Therefore, $\ell_R(R/I_n) \approx n^{d}$ for $d = \dim R/I = \dim L$ by applying \cref{lemma}. 

As for the converse, given such a lifting system $\{ M_n \}_{\geq 1}$, let $L$ be its associated lift. By \cref{liftingprop}, one has that $L \otimes_R R/\ma^n \cong M_n$ for each $n \geq 1$, and thus by applying \cref{fittingLemma}, one has the containment \cref{containment} as above for each $n \geq 1$. Using our hypothesis and applying \cref{lemma}, we have that $\dim L = \dim R/I = d \geq r$. This completes the proof.
\end{proof}

\begin{remark}
    The converse of \cref{theTheorem}(2) holds if the lifting system considered $\{ M_n \}_{\geq 1}$ satisfies $\ell_R(R/J_n) \approx n^r$ where $J_n = \ann_R(M_n)$; this follows by \cref{Fittingideals}. However, we do not know if the the forward direction also holds. 
\end{remark}
     
\begin{corollary}\label{cor2}
    Given a surjection of noetherian local rings $R \twoheadrightarrow S = R/\ma$ where $R$ is $\ma$-adically complete, if $R$ is regular or $S$ is a perfect $R$-module, then a finite length $S$-module $M$ Serre lifts to $R$ if and only if $M$ admits a lifting system $\{M_n\}_{n \geq 1}$ such that $\ell_R(R/\Fitt_0^R(M_n)) \approx n^{\dim R - \dim S}$.  
\end{corollary}

\begin{proof}
    If $M$ Serre lifts to an $R$-module $L$, then
    \[\dim{L} = \dim{R}-\dim{S},\]
    and so the lifting dimension of $M$ is at least $\dim{R}-\dim{S}$, which by \cref{theTheorem}(2) implies $M$ admits a lifting system with the required property on Fitting ideals.

    Conversely, suppose that $M$ admits a lifting system $\{M_n\}_{n \geq 1}$ such that 
    \[\ell_R(R/\Fitt_0^R(M_n)) \approx n^{\dim R - \dim S}.\]
    We need to show that $\dim{L}$ is equal to $\dim{R}-\dim{S}$. When $R$ is regular, we have
    \[\dim{L} \leq \dim{R}-\dim{S} \leq \dim{L}\]
    where the first inequality is by Serre's dimension inequality \cite[Ch.V, Thm.3]{Serre} and the second inequality is by \cref{theTheorem}(2) taking $d=r=\dim{R}-\dim{S}$. Now assume that $S$ is perfect. The following inequalities always hold
    \[\operatorname{grade}_{R}S \leq \height\ann_{R}(S) \leq \dim{R} - \dim{S} \leq \pdim_{R}S.\]
    The third inequality is the Intersection theorem of Peskine and Szpiro \cite{Hochster:1975b, Peskine/Szpiro:1973, RobertsNIT:1987, Roberts:1989}, and see \cite[(2.3.3)]{Avramov/Foxby:1998} for the first two. Since $S$ is perfect, the grade and projective dimension of $S$ are equal, giving the equality below
    \[\dim{R}-\dim{S} \leq \dim{L} \leq \pdim_{R}S = \dim{R}-\dim{S}.\]
    The first inequality is by \cref{theTheorem}(2) taking $d=r=\dim{R}-\dim{S}$ and the second inequality is by the Intersection theorem of Peskine and Szpiro.
\end{proof}

In the following example we use the results of this section to compute the lifting dimension of $M$ from \cref{workingexample}.

\begin{example}
We go back to \cref{workingexample} again. The calculation in \cref{workingexamplefordepth} shows that $M$ lifts to $R$ and thus by \cite[Proposition 1.2]{KC/SotoLevins:2024} we know that $M$ also Serre lifts to $R$, i.e. $\ldim{\varphi}{M} = \ldepth{\varphi}{M} =2.$

However, for demonstration, we compute $\ldim{\varphi}{M}$ by using the theorem above. Note that  
\[ I_n = \text{Fitt}_{0}^{R}(M_n) = (z^2-st)+ \ma^n(s, t, z, \ma^n) \subseteq (z^2 -st) + \ma^n, \] 
 and thus observe that \[ \ell_R(R/I_n) \geq \ell_R\left(\frac{k[[x, y, z]]}{(x^2, y^2)^n + (z^2-st)} \right)= 2 \ell_R\left(\frac{k[[x, y]]}{(x^2, y^2)^n} \right) \approx n^2.\]
Therefore, \[ 2 \leq \ldim{\varphi}{M} \leq 2,\]where the first inequality is due to the theorem above and the second is due to \cref{ldimbound}. 
 \end{example}

\section{Applications and examples} \label{S_Examples}
Utilizing the methods introduced in this paper, in this section we produce examples of unliftable modules of finite length and finite projective dimension that are Serre liftable and identify a sufficient condition for cyclic modules to intersect properly over a regular ring; see \cref{E_UnliftableSerreLiftbleExample} and \Cref{T_IntersectProperlyEquimultiple} respectively.

\subsection{Unliftable modules with Serre lifts.} \label{unliftablemodules}
It is an interesting problem to produce Cohen--Macaulay modules of finite projective dimension that do not come from a regular ring; that is, modules that do not lift to a regular ring. There exists an unliftable finite length module of finite projective dimension over a two dimensional quadratic hypersurface that admits a Serre lift to a regular ring \cite[Example 3.2]{KC/SotoLevins:2024}. In this section, we produce a family of unliftable finite length modules of finite projective dimension over local complete intersection rings of arbitrary (Krull) dimension and codimension at least two and these modules admit Serre lifts to a regular ring; see \Cref{E_UnliftableSerreLiftbleExample}. 

Throughout this subsection, we assume the following setup: 
    Let $k$ be a field and $R=k[[x_1,...,x_d,y_1,...,y_d]]$ for some $d\geq 2$. Consider the ideal $\ma = (g_1,...,g_d)$ where $g_i = x_i^d+y_i^d$ for $1 \leq i \leq d$, and set $S=R/\ma$. Let $\varphi\colon R\rightarrow S$ be the quotient map. Note that $S$ is a $d$-dimensional local complete intersection ring of codimension $d$.

The sequences $x_1^{d-1}, \ldots, x_d^{d-1}$ and $y_1^{d-1}, \ldots, y_d^{d-1}$ both form maximal regular sequences in $S$. Therefore, we have a well-defined $S$-linear map  
\[ \eta\colon(x_1^{d-1},\ldots,x_d^{d-1})\xrightarrow{}S/(y_1^{d-1},\ldots,y_d^{d-1}),\] 
where $\eta(x_i^{d-1})=-f_i$, with $f_i=\prod_{j \not=i}x_j$ for $1 \leq i \leq d$. To see that $\eta$ is a well defined $S$-linear map, let $e_{1},\dots,e_{d}$ be a basis for $S^{d}$. Note that the kernel of the map $S^{d}\rightarrow S/(y_1^{d-1},\ldots,y_d^{d-1})$ defined by $e_{i}\mapsto -f_{i}$ contains the Koszul relations $x_i^{d-1}e_{j} - x_j^{d-1}e_{i}$, and since $x_1^{d-1},\ldots,x_d^{d-1}$ is an $S$-regular sequence, the Koszul relations are the only relations on $(x_1^{d-1},\ldots,x_d^{d-1})$.

Let $M$ denote the $S$-module formed by the following pushout 
\[\begin{tikzcd}
	{(x_1^{d-1},\ldots,x_d^{d-1})} & S \\
	{S/(y_1^{d-1},\ldots,y_d^{d-1})} & M.
	\arrow[hook, from=1-1, to=1-2]
	\arrow["\eta", from=1-1, to=2-1]
	\arrow[from=1-2, to=2-2]
	\arrow[from=2-1, to=2-2]
\end{tikzcd}\]

\begin{proposition} \label{P_UnliftableSerreLiftbleResult}
    $M$ is an $S$-module of finite length and finite projective dimension. Furthermore, the following is a presentation matrix for $M$ as an $S$-module. \[ \phi = \begin{pmatrix}
        0 &\ldots & 0 & x_1^{d-1}&\ldots & x_d^{d-1}\\
        y_1^{d-1}& \ldots & y_d^{d-1} & f_1 &\ldots &f_d 
    \end{pmatrix}\] 
\end{proposition}
\begin{proof}
    Given its construction, $M$ is a representative of a class in 
    \[ \Ext^1_S\left(\frac{S}{(x_1^{d-1}, \ldots, x_d^{d-1})}, \frac{S}{(y_1^{d-1}, \ldots, y_d^{d-1})}\right),\] 
    see \cite[Lemma 7.28]{Rotman:2009}. Thus, the first assertion follows as these modules are quotients by maximal regular sequences in $S$ and hence are of finite projective dimension. The assertion about the presentation matrix follows by the pushout construction.
\end{proof}

Our goal is to prove the following. 
\begin{theorem} \label{E_UnliftableSerreLiftbleExample}
    Consider $\varphi\colon R \twoheadrightarrow S$ and the finite length $S$-module $M$ as above. 
    \begin{enumerate}
        \item \label{part1} There exists a lifting system $\{M_n\}_{n\geq 1}$ along $\varphi$ for $M_1 = M$ such that 
        \[ \ell_R(R/\Fitt_0^R(M_n)) \approx n^d.\] 
        In particular, $M$ Serre lifts to $R$. 
        \item \label{part2} Every lifting system $\{M_n\}_{n\geq 1}$ for $M_1 = M$ satisfies 
        \[ \varprojlim \Tor_1^R(M_n, S) \not= 0.\] 
        In particular, $M$ does not lift to $R$.
    \end{enumerate}
\end{theorem}

Before proving the theorem, we first need a technical lemma needed for \cref{E_UnliftableSerreLiftbleExample}\cref{part2}. Below, we write $(N)_j$ for the $j$th column of a matrix $N$. 
\begin{lemma} \label{L_BenKesavan} 
    Consider $\varphi\colon R \twoheadrightarrow S$ and the finite length $S$-module $M$ as above. Let $\{M_n\}_{n \geq 1}$ be a lifting system for $M$ along $\varphi$, and following the notation of \cref{thisiswhatitlookslike}, let $\phi_n$ be a presentation matrix of $M_n$ as an $R$-module for $n \geq 1$.
    
    Let $h = \prod_{i \leq {d-2}}x_i \in R$. For each $n \geq 0$, there exists 
    \[m_n = (m_{n, 1}, \ldots, m_{n, d}) \in (R^{\oplus 2})^{\oplus d}\] 
    such that \begin{enumerate}
        \item The following equality holds over $R^{\oplus 2}$ 
        \[\sum_{i=1}^d g_i m_{n, i} = -hy_{d-1}(\phi_{n+1})_{d-1} + hy_d(\phi_{n+1})_d + x_d^{d-1}(\phi_{n+1})_{2d-1} - x^{d-1}_{d-1}(\phi_{n+1})_{2d}. \]
        \item $m_{n, d-1}, m_{n, d} \in \m^{d-2}R^{\oplus 2} \setminus \m^{d-1}R^{\oplus 2},$ and
        \item $m_{n+1} - m_n \in \ma^{n}(R^{\oplus 2})^{\oplus d}$.
    \end{enumerate}
\end{lemma}

\begin{proof}
    Due to \cref{thisiswhatitlookslike}, for each $j \geq 1$, we can write, \begin{equation} \label{decomp}
    \sigma_j = \sum_{k=1}^d \sigma_{j,k} g_k
\end{equation} 
where $\sigma_{j, k}$ is a matrix with each entry in $\ma^{j-1}$. Let $h = \prod_{i \leq {d-2}}x_i$ and define for each $n \geq 1$ and for $ 1 \leq i \leq d$, 
    \[ r_{n, i} = -hy_{d-1}(\sigma_{n, i})_{d-1} + hy_d(\sigma_{n, i})_d + x_d^{d-1}(\sigma_{n, i})_{2d-1} - x^{d-1}_{d-1}(\sigma_{n, i})_{2d} \in R^{\oplus 2},\]
and note that $r_{n,i} \in \ma^{n-1} R^{\oplus 2}$ due to \cref{decomp}. Let $m_{0, d-1} = \begin{pmatrix}
    0 \\ -h
\end{pmatrix}$, $m_{0, d} = \begin{pmatrix}
    0 \\ h
\end{pmatrix}$, and $m_{0, i} = \begin{pmatrix}
    0 \\ 0
\end{pmatrix}$ for $i \not= d-1, d$. Now let us define for each $n \geq 1$, 
\[m_{n, i} = m_{0, i} + \sum_{j = 1}^{n} r_{j, {i}},  \text{for }1 \leq i \leq  d.\]
Since $h \in \m^{d-2} \setminus \m^{d-1}$, \cref{L_BenKesavan}(2) holds by definition. Also by definition, we have $m_{n+1} = m_{n} + r_{n+1}$ for each $n \geq 0$, where 
    \[r_{n+1} = (r_{n+1, 1}, \ldots, r_{n+1, d}) \in \ma^{n}(R^{\oplus 2})^{\oplus d}.\]
    This proves \cref{L_BenKesavan}(3). 
    
    We now prove \cref{L_BenKesavan}(1) by induction on $n$. The $n = 0$ case is the following identity that can be checked directly where $\phi = \phi_1$. 
    \[ g_{d-1}\begin{pmatrix}
    0 \\-h
\end{pmatrix} +g_d \begin{pmatrix}
    0\\h
\end{pmatrix} = -hy_{d-1}(\phi)_{d-1} + hy_d(\phi)_d + x_d^{d-1}(\phi)_{2d-1} - x^{d-1}_{d-1}(\phi)_{2d}.\] 
Now for $n \geq 1$, we have
\[
\begin{aligned}
\sum_{i=1}^d g_i m_{n, i}
&= \sum_{i=1}^d g_i m_{n-1, i} + \sum_{i=1}^d g_i r_{n, i}       \\
&= -hy_{d-1}((\phi_{n})_{d-1} + (\sigma_{n})_{d-1}) + hy_{d}((\phi_{n})_{d} + (\sigma_{n})_{d}) \\
& \hspace{4mm}+ x_d^{d-1}((\phi_{n})_{2d-1} + (\sigma_{n})_{2d-1}) -x_{d-1}^{d-1}((\phi_{n})_{2d} + (\sigma_{n})_{2d})   \\
&= -hy_{d-1}(\phi_{n+1})_{d-1} + hy_d(\phi_{n+1})_d + x_d^{d-1}(\phi_{n+1})_{2d-1} - x^{d-1}_{d-1}(\phi_{n+1})_{2d}.
\end{aligned}
\]
The first equality is due to the fact that $m_n = m_{n-1} + r_n$, the second is by induction and the definition of $r_n$, and the third is by the fact that $\phi_{n+1} = \phi_{n} + \sigma_n$. 
\end{proof}

\begin{proof}[Proof of \Cref{E_UnliftableSerreLiftbleExample}]
    For both \cref{part1} and \cref{part2}, the ``in particular" part follows from  \cref{cor2} and \cref{cor1} respectively. 

    Assume the notation from \cref{thisiswhatitlookslike} and consider the lifting system $\{M_n\}_{n\geq 1} $ where $\sigma_j = 0$ for all $j \geq 1$. Thus, as an $R$-module, 
    \[ M_n = \coker \begin{pmatrix}
        0 &\ldots & 0 & x_1^{d-1}&\ldots & x_d^{d-1} & A_n & 0\\
        y_1^{d-1}& \ldots & y_d^{d-1} & f_1 &\ldots &f_d&0 & A_n
    \end{pmatrix}.\]
      Clearly, $ \Fitt^R_0(M_n) \subseteq (x_1, \ldots, x_d) + \ma^n$, and thus 
      \[ \ell_R(R/\Fitt^R_0(M_n)) \geq \ell_R(\overline{R}/\ma^n \overline{R}),\]
      where $\overline{R} = R/(x_1, \ldots, x_d)$ is a $d$-dimensional local ring. Thus, for $n$ large enough, $\ell_R(R/\Fitt^R_0(M_n)) \geq n^d$; this follows from the standard Hilbert-Samuel theory. On the other hand, let $L$ be the lift associated to the system and let $I = \Fitt^R_0(L)$. By \cref{liftingprop}(1) and \cref{fittingLemma}, we have the following containments 
\begin{equation*}
    I + \ma^{n\cdot\mu_{R}(L)} \subseteq \Fitt^R_0(M_n) \subseteq I + \ma^n
\end{equation*}
for all $n \geq 1$, and so $\ell_R(R/\Fitt^R_0(M_n)) \approx n^{\dim R/I}$ by applying \cref{lemma}. Since $R$ is regular, Serre's dimension inequality implies
      \[\dim R/I = \dim{L} \leq \dim{R} - \dim{S} = 2d - d = d,\]
      and so the polynomial growth rate of $\ell_R(R/\Fitt^R_0(M_n))$ is at most $n^{d}$, and so our proof for \cref{part1} is complete. 
    
  For the proof of \cref{part2}, let $\{M_n\}_{n\geq 1}$ be a lifting system for $M$ along $\varphi$. Applying \cref{thisiswhatitlookslike}, we have that at the $n$'th stage, we can write a presentation matrix of $M_n$ as an $R$-module as follows 
\[ \phi_n = \left(
\begin{array}{c|cc}
\multirow{2}{*}{$\phi + \sum_{j=1}^{n-1} \sigma_j$} & A_n& 0 \\
                       & 0   & A_n
\end{array}\right),\] 
where for each $j \geq 1$, $\sigma_j$ is a matrix the same size as $\phi$ and its entries are either zero or in $\ma^j \setminus \ma^{j+1}$. For each $n \geq 1$, let \[ K_n = \Kos_R(g_1, \ldots, g_d) \otimes_R M_{n} \] 
with the differential $\del_n$ where $\Kos_R(g_1, \ldots, g_d)$ is the Koszul complex on $g_1, \ldots, g_d$ over $R$. For each $n \geq 1$, let $\eta_{n} = \text{cl}(m_{n})$ denote the class of $m_{n}$ in $M_{n}^{\oplus d} = (K_n)_1$. Observe that 
\[ \begin{aligned}
    \del_n(\eta_n) &= \sum_{i=1}^d g_i m_{n, i} = \sum_{i=1}^d g_i m_{n-1, i} + \sum_{i=1}^d g_i r_{n, i} \\
    &= -hy_{d-1}(\phi_{n})_{d-1} + hy_d(\phi_{n})_d + x_d^{d-1}(\phi_{n})_{2d-1} - x^{d-1}_{d-1}(\phi_{n})_{2d} + \sum_{i=1}^d g_i r_{n, i}\\
    &=0 + \sum_{i=1}^d g_i r_{n, i} \in \ma^{n}(R^{\oplus2}) \\
    &=0,
\end{aligned} \] 
where the $r_{n,i}$ are as in \Cref{L_BenKesavan}, where the first equality is by the definition of $K_n$, the third is by \cref{L_BenKesavan}(1), and the fourth is because the sum is a linear combination of columns of the presentation matrix of $M_n$ and is thus zero. 

Therefore, $\eta_n$ defines a nonzero cycle in $K_n$. Viewed as a matrix, the only nonzero elements in the differential $\del_n$ are the $g_i$'s which lie in $\m^{d}$, so in conjunction with \cref{L_BenKesavan}(2), we see that $\eta_n$ defines a nonzero element in homology. Thus, 
\[ 0\not=\eta_n \in \H_1(K_n) \cong \Tor^R_1(M_n, S),\]
where the isomorphism is due to the fact that the $g_i$'s form a regular sequence in $R$ and that $\Kos_R(g_1, \ldots, g_d)$ resolves $S$ over $R$. Finally, along the natural surjection, $M_{n+1} \twoheadrightarrow M_{n+1}/\ma^n M_{n+1} = M_n$, due to \cref{L_BenKesavan}(3), $\eta_{n+1} \mapsto \eta_n$. This shows that 
\[ 0 \not= \varprojlim \eta_{n} \in \varprojlim \Tor^R_1(M_n, S),\] 
which completes the proof of the theorem. 
\end{proof}

\subsection{Serre liftable cyclic modules.}
Given a local ring $R$, two finitely generated $R$-modules $M$ and $N$ such that $\ell_R(M \otimes_RN)<\infty$ are said to \textit{intersect properly} if \[ \dim M + \dim N = \dim R.\]
If, for example, $R$ is a regular local ring, one always has that \[ \dim M + \dim N \leq \dim R, \] and it is interesting to find sufficient conditions when equality holds. In this direction, our methods yield the following result. 
\begin{theorem} \label{T_IntersectProperlyEquimultiple}
    Suppose that $(R, \m)$ is a regular local ring and $\ma$ and $\mathfrak b$ are two ideals such that $\sqrt{\ma + \mathfrak b} = \m$. If $\ma$ is an equimultiple ideal and $\mathfrak b$ does not contain any minimal generator of $\ma^n$ for some $n$ sufficiently large, then \[ \dim R/\ma + \dim R/\mathfrak b = \dim R.\]
\end{theorem}

    Equivalently, the conclusion of this theorem says that $R/\mathfrak b$ is a Serre lift of the finite length $S$-module $S/\mathfrak bS$ where $S = R/\ma$. We give a proof of this theorem below. First we recall what it means for an ideal to be equimultiple, and then prove a lemma.
\begin{chunk}
    Let $R$ be a noetherian local ring and $\ma \subseteq R $ an ideal. The \textit{analytic spread} of $\ma$, denoted $\ell(\ma)$, measures the polynomial rate in which the $\mu(\ma^n)$ grows for $n \geq 1$, i.e. we have $\mu(\ma^n) \approx n^{\ell(\ma)-1}$, see \cite[Definition 4.6.7]{Bruns/Herzog:1998}. It is always true that
    \[\height{\ma} \leq \ell(\ma) \leq \dim{R},\]
    see \cite[Exercise 4.6.13.c]{Bruns/Herzog:1998}, and when the first inequality is an equality, we say $\ma$ is an \textit{equimultiple} ideal.

    Given an ideal $\mathfrak b \subseteq R$, by the Artin-Rees lemma, 
    \[ \text{ar}_{\ma}(\mathfrak b) =\inf\{s \colon \text{for } n \geq s, \mathfrak b \cap \ma^{n} = \ma^{n-s}(\mathfrak b \cap \ma^s)\}\] is a finite integer. Therefore, $\mathfrak b \cap \ma^n \subseteq \m \ma^n$ for some $n \geq \text{ar}_{\ma}(\mathfrak b)$  is equivalent to the condition $\mathfrak b \cap \ma^n \subseteq \m \ma^n$ for all $n \gg 0$. 
\end{chunk}

\begin{lemma} \label{T_XiaolingTheorem}
    Let $(R, \m)$ be a noetherian local ring, let $\ma \subseteq R$ be an ideal such that $R$ is $\ma$-adically complete, and consider the surjection $\varphi\colon R \twoheadrightarrow S =R/\ma$. Suppose that $\mathfrak b \subseteq R$ is an ideal such that $ \sqrt{\ma  + \mathfrak b} = \m$ and that the containment 
    \[ 
    \mathfrak b \cap \ma^n \subseteq \m \ma^n
    \] 
    holds for some $n$ sufficiently large. Setting $I_n = \mathfrak b + \ma^n$ we have that $\{R/I_n\}_{n \geq 1}$ is a lifting system for the $S$-module $S/\mathfrak b S$ along $\varphi$ such that $\ell_R(R/I_n) \approx n^d$ for $d \geq  \ell(\ma)$. In particular, $\dim R/\mathfrak b \geq \ell(\ma)$.
\end{lemma}

\begin{proof}
    It follows from the definition that the collection $\{ R/I_n\}_{n \geq 1}$ is a lifting system for the $S$-module $S/\mathfrak b S$. We now coarsely estimate $\ell_R(I_n/I_{n+1})$ for $n \geq 1$ from below by computing its minimal number of generators. We claim that the following isomorphisms hold for each $n \geq 1$. 
    \[\frac{I_n/I_{n+1}}{\m (I_n/I_{n+1})} = \frac{I_n}{\m I_n + I_{n+1}} \cong \frac{\ma^n}{\m \ma^n + (\mathfrak{b} \cap\ma^n)}. \]
    The equality is clear. To see the isomorphism, notice that since $\m \mathfrak{b} \subseteq \mathfrak{b}$ and since $\ma^{n+1} \subseteq \m \ma^n$, we have
    \[ \frac{I_n}{\m I_n + I_{n+1}} = \frac{\mathfrak{b} + \ma^n}{\m(\mathfrak{b} + \ma^n) + \mathfrak{b} + \ma^{n+1}} = \frac{\mathfrak{b} + \ma^n}{\mathfrak{b} + \m \ma^n }.\]
    Also, the map 
    \[ \ma^n \twoheadrightarrow \frac{\mathfrak{b} + \ma^n}{\mathfrak{b} + \m \ma^n}\]
    defined by $r \mapsto r +  \mathfrak{b} +\m \ma^n$ is surjective, and the kernel of this map is 
    \[\ma^n \cap (\mathfrak{b} + \m \ma^n) = (\ma^n \cap \mathfrak{b}) + \m \ma^n\] 
    since $\m \ma^n \subseteq \ma^n$; this last equality follows from the modular law, see \cite[Page 6]{Atiyah/Macdonald:1969}. This proves the claim.
    
    Now from the claim and the following exact sequence
    \[0 \to \frac{\m \ma^n + (\mathfrak{b} \cap \ma^n)}{\m \ma^n} \to \frac{\ma^n}{\m \ma^n} \twoheadrightarrow \frac{\ma^n}{\m \ma^n + (\mathfrak{b} \cap \ma^n)} \to 0,\]
    it follows that 
    \[\mu_R(I_n/I_{n+1}) = \mu(\ma^n) - \dim_k\frac{\m \ma^n +(\mathfrak{b} \cap \ma^n)}{\m \ma^n}.\] 
    Since the containment $\mathfrak{b} \cap \ma^{n} \subseteq \m \ma^{n}$ holds for $n$ large enough, we have 
    \[\mu_R(I_n/I_{n+1}) = \mu(\ma^n). \]
    From the following exact sequence 
    \[ 0 \to I_n/I_{n+1} \to R/I_{n+1} \to R/I_n \to 0\]
    we see that
    \[\ell_R(R/I_{n+1}) - \ell_R(R/I_n) = \ell_R(I_n/I_{n+1}) \geq \mu_R(I_n/I_{n+1}) = \mu(\ma^n) \approx n^{\ell(\ma)-1}.\]
    This tells us $\ell_R(R/I_n)  \geq n^{\ell(\ma)}$ for $n$ large enough. Therefore, $\ell_R(R/I_n) \approx n^d$ with $d \geq \ell(\ma)$. The associated lift of the system is $R/\mathfrak b$ since $R$ is $\ma$-adically complete, and so, by \cref{theTheorem}(2), $\dim R/\mathfrak b \geq \ell(\ma)$. Just recall that the zeroth Fitting ideal of $R/I_{n}$ is $I_{n}$.
\end{proof}

    \begin{proof}[Proof of \cref{T_IntersectProperlyEquimultiple}]
    For any $n \geq 1$, the condition that $\mathfrak b$ does not contain any minimal generator of $\ma^n$ is equivalent to the inclusion $\mathfrak b \cap \ma^n \subseteq \m \ma^n$. Now just note that we have
    \[\dim{R/\mathfrak b} \geq \ell(\ma) = \height \ma = \dim{R} -\dim{R/\ma} \geq \dim{R/\mathfrak b}\]
    where the first inequality is by \Cref{T_XiaolingTheorem}, the first equality is by $\ma$ being an equimultiple ideal, and the second inequality is by Serre's dimension inequality.
    \end{proof}

\begin{appendix}

\section{Proof of \cref{liftingprop}} \label{A_ProofLiftingSystemProp}
In this section we prove \cref{liftingprop}. The proof of the second part of this proposition is an abstraction of the proof of \cite[Theorem 1.2]{Auslander/Ding/Solberg:1993}. First we recall the Mittag-Leffler condition, see \cite[Definition 3.5.9]{Christensen/Foxby/Holm:2024}.
\begin{chunk} An inverse system of modules $\{\lambda_{n+1}\colon M_{n+1}\rightarrow M_{n}\}_{n\geq 1}$ is said to satisfy the Mittag-Leffler condition if for every $n\geq 1$ the descending chain
\[\text{im}(\lambda_{n+1})\supseteq \text{im}(\lambda_{n+1}\lambda_{n+2}) \supseteq \text{im}(\lambda_{n+1}\lambda_{n+2}\lambda_{n+3})\supseteq \dots\]
of submodules of $M_{n}$ stabilizes. Clearly, systems $\{\lambda_{n+1}\colon M_{n+1}\rightarrow M_{n}\}_{n\geq 1}$ consisting of artinian modules and systems such that $\lambda_{n+1}$ is surjective for all $n$ satisfy the Mittag-Leffler condition.
\end{chunk}

\begin{proof}[Proof of \cref{liftingprop}]
Below we give a proof for the case $n = 1$. But for each $n \geq 1$, running the same proof for $\ma^n$ shows that the module $L_n = \varprojlim_{i \geq n} M_i$ satisfies $L_n \otimes_R R/\ma^n \cong M_n$. However, it is clear that $L \cong L_n$ for each $n \geq 1$, so we have the desired conclusion.
    
Let $\ma =  (a_1, \ldots, a_{\mu(\ma)})$. Consider the map 
\[\alpha_{n}\colon M_n^{\oplus \mu(\ma)} \xrightarrow{[a_1 \cdots a_{\mu(\ma)}]}M_n\]
for all $n \geq 1$, and let $K_{n} = \ker(\alpha_{n})$. Note that $\coker(\alpha_{n}) = M_{n}/\ma M_n \cong M$ for all $n \geq 1$. For (1), consider the following commutative diagram with exact rows

\[\begin{tikzcd}
	& \vdots & \vdots & \vdots & \vdots \\
	0 & {K_{n+1}} & {M_{n+1}^{\oplus\mu(\ma)}} & {M_{n+1}} & M & 0 \\
	0 & {K_{n}} & {M_{n}^{\oplus\mu(\ma)}} & {M_{n}} & M & 0 \\
	& \vdots & \vdots & \vdots & \vdots
	\arrow[from=1-2, to=2-2]
	\arrow[from=1-3, to=2-3]
	\arrow[from=1-4, to=2-4]
	\arrow[from=1-5, to=2-5]
	\arrow[from=2-1, to=2-2]
	\arrow[from=2-2, to=2-3]
	\arrow[from=2-2, to=3-2]
	\arrow["{\alpha_{n+1}}", from=2-3, to=2-4]
	\arrow[from=2-3, to=3-3]
	\arrow[from=2-4, to=2-5]
	\arrow[from=2-4, to=3-4]
	\arrow[from=2-5, to=2-6]
	\arrow[from=2-5, to=3-5]
	\arrow[from=3-1, to=3-2]
	\arrow[from=3-2, to=3-3]
	\arrow[from=3-2, to=4-2]
	\arrow["{\alpha_{n}}", from=3-3, to=3-4]
	\arrow[from=3-3, to=4-3]
	\arrow[from=3-4, to=3-5]
	\arrow[from=3-4, to=4-4]
	\arrow[from=3-5, to=3-6]
	\arrow[from=3-5, to=4-5]
\end{tikzcd}\]
This diagram gives rise to two more diagrams
\[\begin{tikzcd}
	& \vdots & \vdots & \vdots \\
	0 & {K_{n+1}} & {M_{n+1}^{\oplus\mu(\ma)}} & {\ma M_{n+1}} & 0 \\
	0 & {K_{n}} & {M_{n}^{\oplus\mu(\ma)}} & {\ma M_{n}} & 0 \\
	& \vdots & \vdots & \vdots
	\arrow[from=1-2, to=2-2]
	\arrow[from=1-3, to=2-3]
	\arrow[from=1-4, to=2-4]
	\arrow[from=2-1, to=2-2]
	\arrow[from=2-2, to=2-3]
	\arrow[from=2-2, to=3-2]
	\arrow[from=2-3, to=2-4]
	\arrow[from=2-3, to=3-3]
	\arrow[from=2-4, to=2-5]
	\arrow[from=2-4, to=3-4]
	\arrow[from=3-1, to=3-2]
	\arrow[from=3-2, to=3-3]
	\arrow[from=3-2, to=4-2]
	\arrow[from=3-3, to=3-4]
	\arrow[from=3-3, to=4-3]
	\arrow[from=3-4, to=3-5]
	\arrow[from=3-4, to=4-4]
\end{tikzcd}\]
and
\[\begin{tikzcd}
	& \vdots & \vdots & \vdots \\
	0 & {\ma M_{n+1}} & {M_{n+1}} & M & 0 \\
	0 & {\ma M_{n}} & {M_{n}} & M & 0 \\
	& \vdots & \vdots & \vdots
	\arrow[from=1-2, to=2-2]
	\arrow[from=1-3, to=2-3]
	\arrow[from=1-4, to=2-4]
	\arrow[from=2-1, to=2-2]
	\arrow[from=2-2, to=2-3]
	\arrow[from=2-2, to=3-2]
	\arrow[from=2-3, to=2-4]
	\arrow[from=2-3, to=3-3]
	\arrow[from=2-4, to=2-5]
	\arrow[from=2-4, to=3-4]
	\arrow[from=3-1, to=3-2]
	\arrow[from=3-2, to=3-3]
	\arrow[from=3-2, to=4-2]
	\arrow[from=3-3, to=3-4]
	\arrow[from=3-3, to=4-3]
	\arrow[from=3-4, to=3-5]
	\arrow[from=3-4, to=4-4]
\end{tikzcd}\]
The second and third diagrams above consist of systems of artinian modules, and so these systems satisfy the Mittag-Leffler condition. Since these systems satisfy the Mittag-Leffler condition, the sequences 
\[0\rightarrow \varprojlim K_{n} \rightarrow \varprojlim M_{n}^{\oplus\mu(\ma)}\rightarrow \varprojlim \ma M_{n}\rightarrow 0\]
and
\[0\rightarrow\varprojlim\ma M_{n} \rightarrow \varprojlim M_{n}\rightarrow \varprojlim M\rightarrow 0\]
are exact by \cite[Theorem 3.5.19]{Christensen/Foxby/Holm:2024}, which gives a four term exact sequence
\[0\rightarrow \varprojlim K_{n} \rightarrow \varprojlim M_{n}^{\oplus\mu(\ma)}\xrightarrow{f} \varprojlim M_{n}\rightarrow \varprojlim M\rightarrow 0.\]
We now claim that the map $f$ is given by $[a_1 \cdots a_{\mu(\ma)}]$. To see this, note that we have a commutative diagram
\[\begin{tikzcd}
	0 & {\varprojlim M_{n}^{\oplus\mu(\ma)}} & {\prod M_{n}^{\oplus\mu(\ma)}} & {\prod M_{n}^{\oplus\mu(\ma)}} & 0 \\
	0 & {\varprojlim M_{n}} & {\prod M_{n}} & {\prod M_{n}} & 0
	\arrow[from=1-1, to=1-2]
	\arrow[from=1-2, to=1-3]
	\arrow["f", from=1-2, to=2-2]
	\arrow["{\Delta^{\oplus}}", from=1-3, to=1-4]
	\arrow["\alpha", from=1-3, to=2-3]
	\arrow[from=1-4, to=1-5]
	\arrow["\alpha", from=1-4, to=2-4]
	\arrow[from=2-1, to=2-2]
	\arrow[from=2-2, to=2-3]
	\arrow["\Delta", from=2-3, to=2-4]
	\arrow[from=2-4, to=2-5]
\end{tikzcd}\]
by \cite[Lemma 3.5.14]{Christensen/Foxby/Holm:2024}, where the maps $\Delta^{\oplus}$ and $\Delta$ are the ones given by \cite[Proposition 3.5.5]{Christensen/Foxby/Holm:2024} and the map $\alpha$ is the map given by $[a_1 \cdots a_{\mu(\ma)}]$. Since limits commute with finite direct sums, the map given by $[a_1 \cdots a_{\mu(\ma)}]$ makes the square on the left commute. Since the map induced on limits is unique (see \cite[Theorem 3.4.5 and Definition 3.4.10]{Christensen/Foxby/Holm:2024}), we must have that $f$ is equal to the map given by $[a_1 \cdots a_{\mu(\ma)}]$, proving the claim.

The claim implies that the image of $f$ is $\ma\varprojlim M_{n}$, giving the third isomorphism below
\[M \cong \varprojlim M \cong (\varprojlim M_{n}) / (\text{im}f) \cong (\varprojlim M_{n}) / (\ma\varprojlim M_{n}) \cong L/\ma L \cong L\otimes_{R} R/\ma.\]
This finishes the proof of (1).

Let us now prove (2). We first show that $L$ is finitely generated. Here we closely follow the proof of the second part of (b) implies (a) of \cite[Theorem 1.2]{Auslander/Ding/Solberg:1993}. Choose a surjection $\psi\colon R^{\oplus t}\rightarrow M$ and consider the following commutative diagram with exact rows
\[\begin{tikzcd}
	& {R^{\oplus t}} \\
	L & M & 0 \\
	{M_{n}} & M & 0
	\arrow["\gamma"', from=1-2, to=2-1]
	\arrow["\psi", from=1-2, to=2-2]
	\arrow[from=2-1, to=2-2]
	\arrow["{\delta_{n}}", from=2-1, to=3-1]
	\arrow[from=2-2, to=2-3]
	\arrow["{=}", from=2-2, to=3-2]
	\arrow[from=3-1, to=3-2]
	\arrow[from=3-2, to=3-3]
\end{tikzcd}\]
where the existence of the map $\gamma$ follows from the fact that $R^{\oplus t}$ is projective, and the maps $L\rightarrow M$ and $\delta_{n}$ are the natural projections. Tensoring this diagram with $R/\ma$ gives the diagram
\[\begin{tikzcd}
	& {R^{\oplus t}}\otimes_{R}R/\ma \\
	L\otimes_{R}R/\ma & M & 0 \\
	{M_{n}\otimes_{R}R/\ma} & M & 0
	\arrow[""', from=1-2, to=2-1]
	\arrow["", from=1-2, to=2-2]
	\arrow[from=2-1, to=2-2]
	\arrow["", from=2-1, to=3-1]
	\arrow[from=2-2, to=2-3]
	\arrow["", from=2-2, to=3-2]
	\arrow[from=3-1, to=3-2]
	\arrow[from=3-2, to=3-3]
\end{tikzcd}\]
Let $C$ be the cokernel of $\delta_{n}\gamma$. Since $M_{n}\otimes_{R}R/\ma=M$ and since the diagram commutes, we have 
\[\delta_{n}\gamma\otimes_{R}R/\ma = \psi\otimes_{R}R/\ma,\]
and so the cokernel of $\delta_{n}\gamma\otimes_{R}R/\ma$ is zero. Since $C/\ma C$ is the cokernel of $\delta_{n}\gamma\otimes_{R}R/\ma$, $C/\ma C$ is zero, and so by Nakayama's lemma $C$ is zero. This implies that the composition $\varphi_{n}:= \delta_{n}\gamma$ is surjective, and so the induced map $\overline{\varphi}_{n}\colon R^{\oplus t}/\ma^{n} R^{\oplus t}\rightarrow M_{n}$ is surjective. From this we get a commutative diagram with exact rows and columns
\[\begin{tikzcd}
	&& 0 & 0 \\
	&& {\ma^{n}R^{\oplus t}/\ma^{n+1}R^{\oplus t}} & {\ma^{n}M_{n+1}} \\
	0 & {\ker{\overline{\varphi}_{n+1}}} & {R^{\oplus t}/\ma^{n+1} R^{\oplus t}} & {M_{n+1}} & 0 \\
	0 & {\ker{\overline{\varphi}_{n}}} & {R^{\oplus t}/\ma^{n} R^{\oplus t}} & {M_{n}} & 0 \\
	&& 0 & 0
	\arrow[from=1-3, to=2-3]
	\arrow[from=1-4, to=2-4]
	\arrow["{\beta_{n+1}}", from=2-3, to=2-4]
	\arrow[from=2-3, to=3-3]
	\arrow[from=2-4, to=3-4]
	\arrow[from=3-1, to=3-2]
	\arrow[from=3-2, to=3-3]
	\arrow["{f_{n}}"', from=3-2, to=4-2]
	\arrow["{\overline{\varphi}_{n+1}}", from=3-3, to=3-4]
	\arrow[from=3-3, to=4-3]
	\arrow[from=3-4, to=3-5]
	\arrow[from=3-4, to=4-4]
	\arrow[from=4-1, to=4-2]
	\arrow[from=4-2, to=4-3]
	\arrow["{\overline{\varphi}_{n}}", from=4-3, to=4-4]
	\arrow[from=4-3, to=5-3]
	\arrow[from=4-4, to=4-5]
	\arrow[from=4-4, to=5-4]
\end{tikzcd}\]
Since $\overline{\varphi}_{n+1}$ is surjective, so is $\beta_{n+1}$, and so the sequence
\[0\rightarrow \ker{\beta_{n+1}}\rightarrow \ker{\overline{\varphi}_{n+1}}\xrightarrow{f_{n}} \ker{\overline{\varphi}_{n}}\rightarrow 0\]
is exact by the snake lemma. We now have a surjective system
\[\begin{tikzcd}
	& \vdots & \vdots & \vdots \\
	0 & {\ker{\overline{\varphi}_{n+1}}} & {R^{\oplus t}/\ma^{n+1} R^{\oplus t}} & {M_{n+1}} & 0 \\
	0 & {\ker{\overline{\varphi}_{n}}} & {R^{\oplus t}/\ma^{n} R^{\oplus t}} & {M_{n}} & 0 \\
	& \vdots & \vdots & \vdots
	\arrow[from=1-2, to=2-2]
	\arrow[from=1-3, to=2-3]
	\arrow[from=1-4, to=2-4]
	\arrow[from=2-1, to=2-2]
	\arrow[from=2-2, to=2-3]
	\arrow[from=2-2, to=3-2]
	\arrow[from=2-3, to=2-4]
	\arrow[from=2-3, to=3-3]
	\arrow[from=2-4, to=2-5]
	\arrow[from=2-4, to=3-4]
	\arrow[from=3-1, to=3-2]
	\arrow[from=3-2, to=3-3]
	\arrow[from=3-2, to=4-2]
	\arrow[from=3-3, to=3-4]
	\arrow[from=3-3, to=4-3]
	\arrow[from=3-4, to=3-5]
	\arrow[from=3-4, to=4-4]
\end{tikzcd}\]
and this gives the following exact sequence by \cite[Theorem 3.5.19]{Christensen/Foxby/Holm:2024} since surjective systems satisfy the Mittag-Leffler condition
\[0\rightarrow \varprojlim \ker{\overline{\varphi}_{n}}\rightarrow \varprojlim R^{\oplus t}/\ma^{n} R^{\oplus t}\rightarrow \varprojlim M_n\rightarrow 0.\]
Since $R$ is $\ma$-adically complete and since $R^{\oplus t}$ is a finitely generated $R$-module, $R^{\oplus t}$ and $\varprojlim R^{\oplus t}/\ma^{n} R^{\oplus t}$ are isomorphic. Since $L=\varprojlim M_{n}$ by definition, $L$ is a quotient of a finitely generated module. This proves that $L$ is finitely generated.

To see the equality $\mu_R(L) = \mu_R(M)$, note that since $L\otimes_{R}S$ is isomorphic to $M$ by the first part of the proposition, we have $\mu_R(L) \geq \mu_R(M)$. Also, in the argument above we can choose $t$ to be $\mu_{R}(M)$, which implies $\mu_R(L) \leq \mu_R(M)$ since $L$ is a quotient of $R^{\oplus t}$. 

Finally, to see why the equality
\[ \ann_R(L) = \bigcap_{i \geq 1}\ann_R(M_n)\]
holds, note that $L$ is a submodule of $\prod M_{n}$, and so $\bigcap\ann_R(M_n)$ is contained in $\ann_R(L)$. For the other containment, suppose $r$ is in $\ann_{R}(L)$ and let $x$ be in $M_{n}$ for some $n$. Choose a tuple $(\dots,x_{2},x_{1})$ in $L$ such that $x_{n}=x$. Such a tuple exists since we started with a surjective system, and this tuple must be annihilated by $r$, which implies $rx_{i}=0$ for all $i$, and so $r$ is in $\ann_{R}(M_{n})$. This finishes the proof of (2), which completes the proof of the proposition.
\end{proof}

\end{appendix}

\section*{Acknowledgments}
We thank Lars Christensen, Eloísa Grifo, Jack Jeffries, Tom Marley, and Mark Walker for the many helpful conversations throughout this project. We also thank Xiaoling He for the key combinatorial ideas that grew into \Cref{T_XiaolingTheorem}. 

\bibliographystyle{amsplain}
\bibliography{references}

@article {Auslander:1961,
    AUTHOR = {Auslander, M.},
     TITLE = {Modules over unramified regular local rings},
   JOURNAL = {Illinois J. Math.},
  FJOURNAL = {Illinois Journal of Mathematics},
    VOLUME = {5},
      YEAR = {1961},
     PAGES = {631--647},
      ISSN = {0019-2082},
   MRCLASS = {13.40 (13.95)},
  MRNUMBER = {179211},
MRREVIEWER = {A.\ Brumer},
       URL = {http://projecteuclid.org/euclid.ijm/1255631585},
}

@phdthesis{KC;2025,
    author = {KC, Nawaj},
    title = {Modules of finite projective dimension and singularities},
    school = {University of Nebraska--Lincoln},
    year = {2025}
}

@InProceedings{Hochster:1981,
author="Hochster, Melvin",
editor="Libgober, Anatoly
and Wagreich, Philip",
title="The dimension of an intersection in an ambient hypersurface",
booktitle="Algebraic Geometry",
year="1981",
publisher="Springer Berlin Heidelberg",
address="Berlin, Heidelberg",
pages="93--106",
isbn="978-3-540-38720-6"
}

@article {RobertsNIT:1987,
    AUTHOR = {Roberts, Paul},
     TITLE = {Le th\'{e}or\`eme d'intersection},
   JOURNAL = {C. R. Acad. Sci. Paris S\'{e}r. I Math.},
  FJOURNAL = {Comptes Rendus des S\'{e}ances de l'Acad\'{e}mie des Sciences.
              S\'{e}rie I. Math\'{e}matique},
    VOLUME = {304},
      YEAR = {1987},
    NUMBER = {7},
     PAGES = {177--180},
      ISSN = {0249-6291},
   MRCLASS = {14C17 (13D25)},
  MRNUMBER = {880574},
MRREVIEWER = {R\"{u}diger\ Achilles},
}

@article {Jorgensen:2003,
    AUTHOR = {Jorgensen, David A.},
     TITLE = {Some liftable cyclic modules},
   JOURNAL = {Comm. Algebra},
  FJOURNAL = {Communications in Algebra},
    VOLUME = {31},
      YEAR = {2003},
    NUMBER = {1},
     PAGES = {493--504},
      ISSN = {0092-7872,1532-4125},
   MRCLASS = {13C13 (13B30)},
  MRNUMBER = {1969236},
MRREVIEWER = {Adela\ N.\ Vraciu},
       DOI = {10.1081/AGB-120016772},
       URL = {https://doi.org/10.1081/AGB-120016772},
}

@incollection {Roberts:1989,
    AUTHOR = {Roberts, Paul},
     TITLE = {Intersection theorems},
 BOOKTITLE = {Commutative algebra ({B}erkeley, {CA}, 1987)},
    SERIES = {Math. Sci. Res. Inst. Publ.},
    VOLUME = {15},
     PAGES = {417--436},
 PUBLISHER = {Springer, New York},
      YEAR = {1989},
      ISBN = {0-387-96990-X},
   MRCLASS = {13H15 (14C17)},
  MRNUMBER = {1015532},
MRREVIEWER = {J.\ K.\ Verma},
       DOI = {10.1007/978-1-4612-3660-3\_23},
       URL = {https://doi.org/10.1007/978-1-4612-3660-3_23},
}

@book{Serre,
    author = {Serre, Jean-Pierre} ,
    title = {Local Algebra},
    publisher = {Springer Berlin, Heidelberg},
    year = {1965},
}

@article {Peskine/Szpiro:1973,
    AUTHOR = {Peskine, C. and Szpiro, L.},
     TITLE = {Dimension projective finie et cohomologie locale.
              {A}pplications \`a la d\'{e}monstration de conjectures de {M}.
              {A}uslander, {H}. {B}ass et {A}. {G}rothendieck},
   JOURNAL = {Inst. Hautes \'{E}tudes Sci. Publ. Math.},
  FJOURNAL = {Institut des Hautes \'{E}tudes Scientifiques. Publications
              Math\'{e}matiques},
    NUMBER = {42},
      YEAR = {1973},
     PAGES = {47--119},
      ISSN = {0073-8301,1618-1913},
   MRCLASS = {14B15 (13D05 13H10)},
  MRNUMBER = {374130},
MRREVIEWER = {Melvin\ Hochster},
       URL = {http://www.numdam.org/item?id=PMIHES_1973__42__47_0},
}

@book {Bruns/Herzog:1998,
    AUTHOR = {Bruns, Winfried and Herzog, J{\"u}rgen},
     TITLE = {{C}ohen-{M}acaulay rings},
    SERIES = {Cambridge Studies in Advanced Mathematics},
    VOLUME = {39},
   EDITION = {Revised},
 PUBLISHER = {Cambridge University Press, Cambridge},
      YEAR = {1998},
     PAGES = {xiv+453},
      ISBN = {0-521-56674-6},
   MRCLASS = {13H10 (13-02)},
  MRNUMBER = {1251956},
MRREVIEWER = {Matthew Miller},
}

@article {Jorgensen:1999,
    AUTHOR = {Jorgensen, David A.},
     TITLE = {Existence of unliftable modules},
   JOURNAL = {Proc. Amer. Math. Soc.},
  FJOURNAL = {Proceedings of the American Mathematical Society},
    VOLUME = {127},
      YEAR = {1999},
    NUMBER = {6},
     PAGES = {1575--1582},
      ISSN = {0002-9939,1088-6826},
   MRCLASS = {13D25 (13C13)},
  MRNUMBER = {1476141},
MRREVIEWER = {Ana\ Jerem\'{\i}as L\'{o}pez},
       DOI = {10.1090/S0002-9939-99-04679-1},
       URL = {https://doi.org/10.1090/S0002-9939-99-04679-1},
}

@book {Eisenbud:1995,
    AUTHOR = {Eisenbud, David},
     TITLE = {Commutative algebra},
    SERIES = {Graduate Texts in Mathematics},
    VOLUME = {150},
      NOTE = {With a view toward algebraic geometry},
 PUBLISHER = {Springer-Verlag, New York},
      YEAR = {1995},
     PAGES = {xvi+785},
      ISBN = {0-387-94268-8; 0-387-94269-6},
   MRCLASS = {13-01 (14A05)},
  MRNUMBER = {1322960},
MRREVIEWER = {Matthew\ Miller},
       DOI = {10.1007/978-1-4612-5350-1},
       URL = {https://doi.org/10.1007/978-1-4612-5350-1},
}

@article {KC/SotoLevins:2024,
    AUTHOR = {KC, Nawaj and Soto Levins, Andrew J.},
     TITLE = {On liftings of modules of finite projective dimension},
   JOURNAL = {Int. Math. Res. Not. IMRN},
  FJOURNAL = {International Mathematics Research Notices. IMRN},
      YEAR = {2024},
    NUMBER = {24},
     PAGES = {14729--14736},
      ISSN = {1073-7928,1687-0247},
   MRCLASS = {13C15 (13H05)},
  MRNUMBER = {4843654},
       DOI = {10.1093/imrn/rnae257},
       URL = {https://doi.org/10.1093/imrn/rnae257},
}

@incollection {Buchsbaum/Eisenbud:1972,
    AUTHOR = {Buchsbaum, David A. and Eisenbud, David},
     TITLE = {Lifting modules and a theorem on finite free resolutions},
 BOOKTITLE = {Ring theory ({P}roc. {C}onf., {P}ark {C}ity, {U}tah, 1971)},
     PAGES = {63--74},
 PUBLISHER = {Academic Press, New York-London},
      YEAR = {1972},
   MRCLASS = {16A64},
  MRNUMBER = {340343},
MRREVIEWER = {J.\ L.\ Dawson},
}

@article {Hochster:1975,
    AUTHOR = {Hochster, Melvin},
     TITLE = {An obstruction to lifting cyclic modules},
   JOURNAL = {Pacific J. Math.},
  FJOURNAL = {Pacific Journal of Mathematics},
    VOLUME = {61},
      YEAR = {1975},
    NUMBER = {2},
     PAGES = {457--463},
      ISSN = {0030-8730,1945-5844},
   MRCLASS = {13C10},
  MRNUMBER = {412169},
MRREVIEWER = {Chr.\ U.\ Jensen},
       URL = {http://projecteuclid.org/euclid.pjm/1102868039},
}

@article {Lichtenbaum:1966,
    AUTHOR = {Lichtenbaum, Stephen},
     TITLE = {On the vanishing of {${\rm Tor}$} in regular local rings},
   JOURNAL = {Illinois J. Math.},
  FJOURNAL = {Illinois Journal of Mathematics},
    VOLUME = {10},
      YEAR = {1966},
     PAGES = {220--226},
      ISSN = {0019-2082},
   MRCLASS = {13.95 (18.20)},
  MRNUMBER = {188249},
MRREVIEWER = {M.\ Auslander},
       URL = {http://projecteuclid.org/euclid.ijm/1256055103},
}

@article {Auslander/Ding/Solberg:1993,
    AUTHOR = {Auslander, Maurice and Ding, Songqing and Solberg, {\O}yvind},
     TITLE = {Liftings and weak liftings of modules},
   JOURNAL = {J. Algebra},
  FJOURNAL = {Journal of Algebra},
    VOLUME = {156},
      YEAR = {1993},
    NUMBER = {2},
     PAGES = {273--317},
      ISSN = {0021-8693},
   MRCLASS = {16D80},
  MRNUMBER = {1216471},
MRREVIEWER = {Dinh Van Huynh},
       DOI = {10.1006/jabr.1993.1076},
       URL = {https://doi.org/10.1006/jabr.1993.1076},
}

@article {Dao:2007,
    AUTHOR = {Dao, Hailong},
     TITLE = {On liftable and weakly liftable modules},
   JOURNAL = {J. Algebra},
  FJOURNAL = {Journal of Algebra},
    VOLUME = {318},
      YEAR = {2007},
    NUMBER = {2},
     PAGES = {723--736},
      ISSN = {0021-8693},
   MRCLASS = {13C05},
  MRNUMBER = {2371969},
MRREVIEWER = {Nguy\cftil{e}n Vi\cfudot{e}t D\~{u}ng},
       DOI = {10.1016/j.jalgebra.2007.09.019},
       URL = {https://doi.org/10.1016/j.jalgebra.2007.09.019},
}

@book {Weibel:1994,
    AUTHOR = {Weibel, Charles A.},
     TITLE = {An introduction to homological algebra},
    SERIES = {Cambridge Studies in Advanced Mathematics},
    VOLUME = {38},
 PUBLISHER = {Cambridge University Press, Cambridge},
      YEAR = {1994},
     PAGES = {xiv+450},
      ISBN = {0-521-43500-5; 0-521-55987-1},
   MRCLASS = {18-01 (16-01 17-01 20-01 55Uxx)},
  MRNUMBER = {1269324},
MRREVIEWER = {Kenneth\ A.\ Brown},
       DOI = {10.1017/CBO9781139644136},
       URL = {https://doi.org/10.1017/CBO9781139644136},
}

@book {Atiyah/Macdonald:1969,
    AUTHOR = {Atiyah, M. F. and Macdonald, I. G.},
     TITLE = {Introduction to commutative algebra},
 PUBLISHER = {Addison-Wesley Publishing Co., Reading, Mass.-London-Don
              Mills, Ont.},
      YEAR = {1969},
     PAGES = {ix+128},
   MRCLASS = {13.00},
  MRNUMBER = {242802},
MRREVIEWER = {Johnny\ A.\ Johnson},
}

@book {Christensen/Foxby/Holm:2024,
    AUTHOR = {Christensen, Lars Winther and Foxby, Hans-Bj{\o}rn  and Holm, Henrik},
     TITLE = {Derived Category Methods in Commutative Algebra},
    SERIES = {Springer Monographs in Mathematics},
 PUBLISHER = {Springer Cham},
      YEAR = {2024},
     PAGES = {vii+1119},
      ISBN = {978-3-031-77452-2},
}

@book {Matsumura:1989,
    AUTHOR = {Matsumura, Hideyuki},
     TITLE = {Commutative ring theory},
    SERIES = {Cambridge Studies in Advanced Mathematics},
    VOLUME = {8},
   EDITION = {Second},
      NOTE = {Translated from the Japanese by M. Reid},
 PUBLISHER = {Cambridge University Press, Cambridge},
      YEAR = {1989},
     PAGES = {xiv+320},
      ISBN = {0-521-36764-6},
   MRCLASS = {13-01},
  MRNUMBER = {1011461},
}

@article {Avramov/Foxby:1998,
    AUTHOR = {Avramov, Luchezar L. and Foxby, Hans-Bj{\o}rn},
     TITLE = {Cohen-{M}acaulay properties of ring homomorphisms},
   JOURNAL = {Adv. Math.},
  FJOURNAL = {Advances in Mathematics},
    VOLUME = {133},
      YEAR = {1998},
    NUMBER = {1},
     PAGES = {54--95},
      ISSN = {0001-8708,1090-2082},
   MRCLASS = {13H10 (13B10 14E40)},
  MRNUMBER = {1492786},
MRREVIEWER = {Rafael\ H.\ Villarreal},
       DOI = {10.1006/aima.1997.1684},
       URL = {https://doi.org/10.1006/aima.1997.1684},
}

@book {Hochster:1975b,
    AUTHOR = {Hochster, Melvin},
     TITLE = {Topics in the homological theory of modules over commutative
              rings},
    SERIES = {Conference Board of the Mathematical Sciences Regional
              Conference Series in Mathematics},
    VOLUME = {No. 24},
      NOTE = {Expository lectures from the CBMS Regional Conference held at
              the University of Nebraska, Lincoln, Neb., June 24--28, 1974},
 PUBLISHER = {Conference Board of the Mathematical Sciences, Washington, DC;
              by American Mathematical Society, Providence, RI},
      YEAR = {1975},
     PAGES = {vii+75},
   MRCLASS = {13DXX (14BXX)},
  MRNUMBER = {371879},
MRREVIEWER = {Tadayuki\ Matsuoka},
}

@book {Rotman:2009,
    AUTHOR = {Rotman, Joseph J.},
     TITLE = {An introduction to homological algebra},
    SERIES = {Universitext},
   EDITION = {Second},
 PUBLISHER = {Springer, New York},
      YEAR = {2009},
     PAGES = {xiv+709},
      ISBN = {978-0-387-24527-0},
   MRCLASS = {18Gxx (13Dxx 16Exx 18-01 20J06)},
  MRNUMBER = {2455920},
MRREVIEWER = {Fernando\ Muro},
       DOI = {10.1007/b98977},
       URL = {https://doi.org/10.1007/b98977},
}
\end{document}